\documentclass[10pt,a4paper]{article}
\usepackage[latin1]{inputenc}
\usepackage{amsmath}
\usepackage{amsthm}
\usepackage{amsfonts}
\usepackage{amssymb}
\usepackage{graphicx}
\usepackage{enumerate}  
\usepackage{stmaryrd}
\usepackage{color}

\newcommand*{\hp}{\emptyset^{\prime}}
\newcommand*{\strings}{\ensuremath{\mathbf{2}^{< \mathbb{N}}}}
\newcommand*{\sequences}{\ensuremath{\mathbf{2}^{\mathbb{N}}}}
\newcommand*{\foralmostall}{\ensuremath{\overset{\infty}{\forall}}}
\newcommand*{\io}{\ensuremath{\overset{\infty}{\exists}}}

\newcommand{\N}{\mathbb{N}}
\newcommand{\Q}{\mathbb{Q}}

\newtheorem*{Theorem*}{Theorem}
\newtheorem{Theorem}{Theorem}[section]
\newtheorem{Lemma}[Theorem]{Lemma}
\newtheorem{Corollary}[Theorem]{Corollary}
\newtheorem{Proposition}[Theorem]{Proposition}

\theoremstyle{definition}
\newtheorem{Definition}[Theorem]{Definition}
\newtheorem*{Definition*}{Definition}

\newtheorem*{Question}{Question}

\theoremstyle{remark}

\begin{document}
	\title{Relativized depth}
	\author{Laurent Bienvenu, Valentino Delle Rose, Wolfgang Merkle}
	\date{}
	\maketitle
	
	\section*{Abstract}
	Bennett's notion of depth is usually considered to describe, roughly speaking, the usefulness and internal organization of the information encoded into an object such as an infinite binary sequence, which as usual will be identified with a set of natural numbers. We consider a natural way to relativize the notion of depth for such sets, and we investigate for various kinds of oracles whether and how the unrelativized and the relativized version of depth differ. 
	
	Intuitively speaking, access to an oracle increases computation power.  Accordingly, for most notions for sets considered in computability theory, for the corresponding classes trivially for all oracles the unrelativized class is contained in the relativized class or for all oracles the relativized class is contained in the unrelativized class. 
	Examples for these two cases are given by the classes of computable and of Martin-L\"{o}f random sets, respectively. However, in the case for depth the situation is different.
	
	It turns out that the classes of deep sets and of sets that are deep relative to the halting set $\hp$ are incomparable with respect to set-theoretical inclusion. On the other hand, the class of deep sets is strictly contained in the class of sets that are deep relative to any given Martin-L\"{o}f-random oracle. The set built in the proof of the latter result can also be used to give a short proof of the known fact that every PA-complete degree is Turing-equivalent to the join of two Martin-L\"{o}f-random sets. In fact, we slightly strengthen this result by showing that every DNC$_2$ function is truth-table-equivalent to the join of two Martin-L\"{o}f random sets.
	
	Furthermore, we observe that the class of deep sets relative to any given K-trivial oracle either is the same as or is strictly contained in the class of deep sets. Obviously, the former case applies to computable oracles. We leave it as an open problem which of the two possibilities can occur for noncomputable K-trivial oracles.
	
	\section{Introduction}
	The notion of depth introduced by Bennett can be seen as a formalization of the idea that the same information can be organized in different ways, making certain encodings more or less useful for certain computational purposes. In particular, depth goes beyond just measuring the information encoded into a finite object by its Kolmogorov complexity, i.e., by the length of an optimal effective description of the object as a binary string. 
	
	Computability theory provides a paradigmatic example of organizing the same information in different ways by the halting set $\hp$ and Chaitin's Omega~$\Omega$. Since $\hp$ is a c.e.~set, describing $\hp \upharpoonright n$, the  prefix of its characteristic sequence of length $n$, requires not more than~$\mathrm{O}(\log n)$ bits. On the other hand, Chaitin~\cite{C} demonstrated that $\Omega$ is ML-random, so describing $\Omega \upharpoonright n$ requires approximately $n$ bits. It is well-known that $\Omega \equiv_T \hp$, that is, $\Omega$ and~$\hp$ can be mutually computed from each other, and in this sense the two sets encode the same information. More specifically, $\Omega$ is a compressed version of $\hp$ since the first $\mathrm{O}(\log n)$ bits of $\Omega$ are sufficient to decide effectively whether $\varphi_n(n)$ halts, i.e., whether~$n$ is in~$\hp$. On the other hand, computing~$\hp$ from~$\Omega$ must necessarily be ``slow'' as one can show that  $\hp$ is not truth-table reducible to $\Omega$, i.e.,  $\hp$ cannot be computed from~$\Omega$ by an oracle Turing machine that runs within some computable time bound. In fact,  $\hp$ is not truth-table reducible to any ML-random set. Elementary proofs of the statements above can be found in \cite{CN}. 	
	
	This situation is captured quite well within the framework of depth. Depth has been introduced by Bennett~\cite{B} in order to distinguish ``useful'' or ``organized'' information from other information such as random noise. In particular, a set is said to be deep if, for any given computable time bound $t$, the difference between the length of the shortest description of the prefix o length~$n$ that can be decoded in time $t(n)$ and the length of its ``true'' optimal description goes to infinity. In other words, no prefix of a deep set can be optimally described within any computable time bound. It turns out that neither ML-random nor computable sets are deep, whereas the halting problem~$\hp$ is deep. Moreover, a key property of depth, known in literature as ``Slow Growth Law'', states that no deep set is truth-table reducible to a nondeep set, in other words, no deep set can be computed from a nondeep set in a ``fast'' way.
	
	In the present work, we consider a natural way of relativizing the notion of depth with the aim of better understanding how an oracle may help in organizing information. This relativization is designed in order to keep focusing on the same class of ``fast'' computations in so far as it is defined in terms of computable time bounds and not in terms of time bounds computable in the oracle. The relativization of depth gains additional interest since it differs in the following respect from most other relativizations considered in computability theory.  Usually, when relativizing a class, trivially for all oracles the relativized class contains the unrelativized class or the other way round, i.e., for all oracles  the unrelativized class contains the relativized one. Examples are given by the classes of computable and of  Martin-L\"{of} random sets, respectively. For the class of deep sets the situation is different since depth is defined in terms of the difference between time-bounded and unbounded Kolmgorov complexity. With access to any given oracle, trivially each individual value of the two latter quantities stays the same or decreases but a priori for any given argument the two corresponding values may decrease by different amounts, hence their difference may increase or decrease. As a consequence, a priori none of the following four cases can be ruled for a given oracle when comparing the classes of deep sets and of sets that are deep relative to the oracle: first, the two classes may be incomparable with respect to set-theoretical inclusion, second, the unrelativized class may be strictly contained in the relativized class, third, the relativized class may be strictly contained in the unrelativized class and, fourth, the two classes may be the same. 
	
	Section \ref{hp-depth} is devoted to proving that the first case applies to the oracle~$\hp$, while in Section~\ref{MLR} we show that the second case applies to all ML-random oracles. As a byproduct of our proof, we slightly strengthen a result due to Barmpalias, Lewis and Ng~\cite{BLN}, which states that every PA-complete degree is the join of two ML-random degrees: in fact, we show that every DNC$_2$ function is truth-table-equivalent to the join of two Martin-L\"{o}f random sets.  In Section~\ref{K-trivial} we observe that every K-trivial oracle  falls under the third or fourth case, while case four holds for all computable oracles. We leave it as an open problem whether the third case holds for some or all noncomputable K-trivial oracles.
	
	\section{Preliminaries}
	This section is devoted to recall known notions from computability theory and algorithmic randomness used throughout the paper. We first introduce our notation, which is quite standard and follows mostly the textbooks \cite{DH}, \cite{N2009} and \cite{S}.
	
	The quantifier $\foralmostall$ is used to mean ``for all but finitely many elements'', while $\io$ means ``there exists infinitely many elements''. For two real-valued functions $f,g$ and a quantifier $Q$, we write 
	$$(Qx)[f(x) \le^+ g(x)] \ (\text{respectively,} \ (Qx)[f(x) \le^\times g(x)]).$$ 
	to mean that there exists a constant $c>0$ such that for all~$x$ in the range of $Q$, $f(x) \leq g(x) + c$ (respectively $f(x) \leq c \cdot g(x)$). 
		
	We denote by $\strings$ the set of all finite binary strings, while $\sequences$ denotes the Cantor space of all infinite binary sequences (that is, the subsets of $\mathbb{N}$). The empty string is denoted by $\lambda$. Given a string $\sigma$, its length is denoted by $|\sigma|$. The set of all strings of length $n$ is denoted by $\mathbf{2}^n$. Given a set $X \in \sequences$, we write $X \upharpoonright n$ to denote $X(0) X(1) \dots X(n-1)$, namely the string of the first $n$ bits of $X$. For a string $\sigma$, the \emph{cylinder} $[\sigma]$ is the set of $X \in \sequences$ such that $X \upharpoonright |\sigma| = \sigma$. Moreover, we denote the Lebesgue measure on $\sequences$ by $\mu$, that is, the unique Borel measure such that $\mu([\sigma])=2^{-|\sigma|}$ for all~$\sigma$. 
	
	Given a partial computable function (equivalently, a Turing machine) $\varphi$ and $\sigma \in \strings$, we write $\varphi(\sigma)[t]$ to denote the output of $\varphi(\sigma)$ after $t$ steps of computation. Moreover, if $\varphi(\sigma)\downarrow = \rho$, we call $\sigma$ a \emph{code for} $\rho$ (\emph{with respect to} $\varphi$).
	
	\subsection{Kolmogorov complexity}
	Recall that a set $A \subseteq \strings$ is \emph{prefix-free} if no member in $A$ is a prefix of another member of the set. A partial computable function $\varphi: \strings \rightarrow \strings$ is \emph{prefix-free} if its domain is a prefix-free set. For the rest of the paper, we fix a prefix-free function $\mathcal{U}$ which is $\emph{universal}$, in the sense that, for every prefix-free partial computable function $\varphi$ there exists $\rho_{\varphi} \in \strings$ such that 
	$$(\forall \sigma) \left[ \mathcal{U}(\rho_{\varphi} \sigma) = \varphi(\sigma) \right].$$ 
	Moreover, we can assume that, if the computation $\varphi(\sigma)$ halts within $t$ steps, then $\mathcal{U}(\rho_{\varphi} \sigma)$ halts (e.g.) within $t^2$ steps (as shown in \cite{HS}).
	
	A classical approach to measure the information contained in some string is given by its \emph{prefix-free complexity}, which is, roughly speaking, the length of its shortest code (with respect to some universal prefix-free partial computable function). We are also interested in \emph{time-bounded} versions of prefix-free complexity. We call a function $t: \mathbb{N} \rightarrow \mathbb{N}$ a \emph{time bound} if it is total and non-decreasing: then the $t$-time-bounded Kolmogorov complexity of a string $\sigma$ is the length of its shortest code running in at most $t(|\sigma|)$ steps.
	
	\begin{Definition} \label{Kcomp}
		The \emph{prefix-free complexity} of $\sigma \in \strings$ is 
		$$K(\sigma) = \min \left\lbrace |\tau| : \ \mathcal{U}(\tau)\downarrow = \sigma \right\rbrace.$$
		
		Given a time bound $t$, the $t$\emph{-time-bounded prefix-free complexity} of $\sigma$ is 
		$$K^t(\sigma)= \min \left\lbrace |\tau| : \ \mathcal{U}(\tau)[t(|\sigma|)]\downarrow = \sigma \right\rbrace.$$
	\end{Definition}
	
Note that while the function $K$ only depends on the choice of the universal machine by an additive constant, this is no longer the case for its time-bounded version. However, the assumption on the universal machine that it can simulate any other machine up to a quadratic blow-up in computation time is enough to make all notions presented in this paper independent from the particular choice of universal machine. 	

	We can also equip our universal prefix-free machine $\mathcal{U}$ with some oracle $A \in \sequences$. Given a string $\sigma$, its prefix-free complexity relative to $A$, denoted by $K^A(\sigma)$, is defined by relativizing Definition \ref{Kcomp} in the obvious way. The same applies to its $t$-time-bounded prefix-free complexity relative to $A$, which we denote by $K^{A,t}(\sigma)$. Recall that $A \in \sequences$ is \emph{Turing reducible} to $B \in \sequences$, and we write $A \le_T B$, if there is an oracle machine $\varphi$ such that $\varphi^B(n)=A(n)$ for all $n$. Moreover, $A$ is \emph{truth-table reducible} (or simply \emph{tt-reducible}) to $B$, and we write $A \le_{tt} B$, if $A \le_T B$ via some oracle machine $\varphi$ such that $\varphi^X$ is total for every oracle $X$. Equivalently, $A \le_{tt} B$ if there is a computable time-bound $t$ and an oracle machine $\varphi$ such that $\varphi^X(n)[t(n)] \downarrow$ for every oracle $X$ and every $n$ and $A \le_T B$ via $\varphi$. The following lemma shows how reductions among sets relate with their relative strength in compressing strings.

	\begin{Lemma} \label{K_lemma}
		Let $A, B \in \sequences$. 
		\begin{enumerate}[(i)]
			\item If $A \leq_T B$, then $K^B(\sigma) \leq^+ K^A(\sigma)$.
			\item If $A \leq_{tt} B$, then for every computable time-bound $t$ there is a computable time-bound $t'$ such that $K^{B,t'}(\sigma) \leq^+ K^{A,t}(\sigma)$.
		\end{enumerate}
	\end{Lemma}
	\begin{proof}
		To prove (i), just observe that any optimal $A$-code $\tau$ for $\sigma$ is also a $B$-code for $\sigma$, as we may consider a Turing machine which first computes the required bits of $A$ using oracle $B$ and then simulates the computation $\mathcal{U}^{A}(\tau)$.
		
		Moreover, whenever $A$ is computable from $B$ in some computable time-bound (that is, $A \leq_{tt} B$), clearly every $A$-$t$-fast-code for $\sigma$ is also a $B$-$t'$-fast-code for $\sigma$, for any computable time-bound $t'$ which allows to compute the required bits of $A$ from $B$.
	\end{proof}

\subsection{Lower-semicomputable discrete semimeasures}

Another way to look at the prefix-free Kolmogorov complexity function, which will be very useful in this paper,  is via \emph{lower-semicomputable discrete semimeasures}. 

	\begin{Definition}
		\begin{enumerate}[(i)]
			\item A \emph{discrete semimeasure} is a function $m: \strings \rightarrow [0, \infty)$ such that $\sum_{\sigma} m(\sigma) \leq 1$. It is \emph{lower-semicomputable} if there is a uniformly computable family of functions $m_s: \strings \rightarrow \mathbb{Q}$ such that, for any string $\sigma$,
			$$(\forall s)[m_{s+1}(\sigma) \ge m_s(\sigma)] \quad \text{and} \quad \lim\limits_{s \rightarrow \infty} m_s(\sigma)=m(\sigma).$$ 
			We will write \emph{lss} for lower-semicomputable discrete semimeasure.
			\item A lss $m$ is \emph{universal} if, for each lss $m'$, 
			$$\left(\forall \sigma \right) \left[ m'(\sigma) \leq^{\times} m(\sigma) \right].$$
		\end{enumerate}
	\end{Definition}

We recall the following known facts about lss. 

\begin{Theorem}[Levin, see paragraph 3.9 of \cite{DH}]
 $ $\\
(i) There exists a universal lss, and from now on we fix one of them which we denote by $\mathbf{m}$. \\
(ii) The function $\sigma \mapsto 2^{-K(\sigma)}$ is a universal lss. Since two universal lss are, by definition, within a multiplicative constant of one another, it follows that $K(\sigma)=^+ - \log \mathbf{m}(\sigma)$. \end{Theorem}

Since $\mathbf{m}$ is lower-semicomputable it can be represented by a non-decreasing family of uniformly computable functions $(\mathbf{m}_s)$. This allows us to define the time-bounded version of $\mathbf{m}$.

\begin{Definition}
Let $t: \mathbb{N} \rightarrow \mathbb{N}$ be a computable time bound. The time-bounded version $\mathbf{m}^t$ of $\mathbf{m}$ is the function defined for all~$\sigma$ by 
\[
\mathbf{m}^t(\sigma)=\mathbf{m}_{t(|\sigma|)}(\sigma)
\]
\end{Definition}

We will constantly make use of the following easy lemma which will allow us to switch between time-bounded Kolmogorov complexity, time-bounded semimeasures and computable semimeasures. 

\begin{Lemma}\label{lem:time-bounded}
For any given computable time bound~$t$, $\mathbf{m}^t$ is a computable semimeasure. Conversely, if $m$ is a computable semimeasure, there exists a computable time bound~$t$ such that $m \leq^\times \mathbf{m}^t$. In particular, for any given time bound~$t$, $2^{-K^t}$ is a computable semimeasure, hence there exists a computable time bound $t'$ such that $2^{-K^t} \leq^\times \mathbf{m}^{t'}$ (or equivalently, $-\log m^{t'} \leq^+ K^t$). 
\end{Lemma}

\begin{proof}
That $\mathbf{m}^t$ is a computable semimeasure is immediate from the definition. Let now $m$ be a computable semimeasure. It is in particular lower-semicomputable hence there is a constant $c>0$ such that $m < c \cdot \mathbf{m}$. Since $\mathbf{m} = \lim_s \mathbf{m}_s$, for all~$n$, it suffices to take $t(n)$ to be the smallest $s$ such that $m(\sigma) < c \cdot \mathbf{m}_s (\sigma)$ for all~$\sigma$ of length~$n$ (it is immediate that~$t$ is computable). The rest of the lemma follows. 
\end{proof}	
	
\subsection{Martin-L\"of randomness}

	We are often interested in the information encoded into the prefixes of some set. The most studied effective randomness notion for sets, \emph{Martin-L\"{o}f randomness}, can be defined in terms of incompressibility of their prefixes. 
	
	\begin{Definition}
		A set $X \in \sequences$ is \emph{Martin-L\"{o}f random} (or simply \emph{ML-random}) if 
		$$\left(\forall n\right)[K(X \upharpoonright n) \ge^+ n].$$
	\end{Definition}	

	We also recall the other two main equivalent approaches to define Martin-L\"{o}f randomness, namely in terms of \emph{ML-tests} lower-semicomputable martingales. 
	
	\begin{Definition}
		\begin{enumerate}[(i)]
			\item A \emph{ML-test} is a sequence of uniformly $\Sigma_1^0$ classes $(\mathcal{U}_n)_{n \in \mathbb{N}}$ such that 
			$$(\forall n)[\mu(\mathcal{U}_n) \le 2^{-n}].$$
			\item A \emph{martingale} is a function $d: \strings \rightarrow [0, \infty)$ such that, for every string $\sigma$, $2d(\sigma) = d(\sigma0) + d(\sigma1)$. 
		\end{enumerate}
	\end{Definition}

	\begin{Proposition}
		For a set $X \in \sequences$, the following statements are equivalent. 
		
		\begin{enumerate}[(i)]
			\item $X$ is ML-random. 
			\item For any ML-test $(\mathcal{U}_n)_{n \in \mathbb{N}}$, 
			$$X \notin \bigcap_{n \in \mathbb{N}} \mathcal{U}_n.$$
			\item For any lower-semicomputable martingale $d$, 
			$$\limsup\limits_{n \rightarrow \infty} d(X \upharpoonright n) < \infty.$$
			\item For any lower-semicomputable martingale $d$, 
			$$\liminf\limits_{n \rightarrow \infty} d(X \upharpoonright n) < \infty.$$
		\end{enumerate}
	\end{Proposition}

	For a proof, see Theorems 6.3.2 and 6.3.4 in \cite{DH}, or Theorem 3.2.9 and Proposition 7.2.6 in \cite{N2009}.
	
	It is well-known that there exists a \emph{universal} ML-test, namely a ML-test $(\mathcal{U}_n)_{n \in \mathbb{N}}$ such that if $X \notin \cap_n \mathcal{U}_n$, then $X$ is Martin-L\"of random (see Theorem 6.2.5 in \cite{DH}, or Fact 3.2.4 in \cite{N2009}). Similarly, there exists a \emph{universal} lower-semicomputable martingale, namely a lower-semicomputable~martingale $\mathbf{d}$ such that $\liminf_n \mathbf{d}(X \upharpoonright n) < \infty$ if and only  if $X$ is Martin-L\"of random (see Corollary 6.3.5 in~\cite{DH}, or Theorem 7.2.8 in~\cite{N2009}).

	Finally, we recall that relativized ML-randomness can be equivalently defined by relativizing all three approaches above in the obvious way.

	\subsection{Depth}
	In many cases, the same information may be organized in different ways, making it more or less useful for various computational purposes. The notion of \emph{depth} was introduced by Bennett as an attempt to separate useful and organized information from random noise and trivial information.
	
	\begin{Definition}
		$X \in \sequences$ is \emph{deep} if, for every computable time-bound $t: \mathbb{N} \rightarrow \mathbb{N}$ 
		
		\[
			\lim\limits_{n \rightarrow \infty} K^t(X \upharpoonright n) - K(X \upharpoonright n) = +\infty.
		\]

		A set which is not deep is called \emph{shallow}.
	\end{Definition}

	By Lemma~\ref{lem:time-bounded}, we can equivalently define a set $X$ to be deep if and only if, for every computable time-bound $t$, 
	$$\lim\limits_{n \rightarrow \infty} \frac{\mathbf{m}(X \upharpoonright n)}{\mathbf{m}^t(X \upharpoonright n)} = + \infty,$$
	or also, if and only if, for every computable semimeasure $m$, 
	$$\lim\limits_{n \rightarrow \infty} \frac{\mathbf{m}(X \upharpoonright n)}{m(X \upharpoonright n)} = +\infty.$$
	
	Let us recall the main properties of logical depth: for proofs of the statements below, see \cite{JLL}. 
	
	\begin{Proposition}
		\begin{enumerate}[(i)]
			\item (Slow Growth Law) Let $X \leq_{tt} Y$ and $X$ be deep. Then $Y$ is deep. 
			\item Every ML-random set is shallow. 
			\item Every computable set is shallow. 
			\item The halting problem $\hp$ is deep.
		\end{enumerate}
	\end{Proposition}


\section{Relativized depth}

	We propose a relativized notion of depth, in order to better understand the power of oracles in organizing information. 
	
	\begin{Definition}
		Given an oracle $A \in \sequences$, we say that $X \in \sequences$ is $A$-\emph{deep} if, for every computable time-bound $t$, 
		\[
		 	\lim\limits_{n \rightarrow \infty} K^{A,t}(X \upharpoonright n) - K^A(X \upharpoonright n) = +\infty
		\] 
		Otherwise, we say that $X$ is $A$-\emph{shallow}.
	\end{Definition} 

	The choice of focusing only on computable time-bounds, instead of considering also $A$-computable ones, in the above definition is mainly due to obtain a relativized version of the Slow Growth Law which keeps working with $tt$-reductions. In other words, we want to stick to ``fast'' oracle computations.
 
	Indeed, we notice that Lemma \ref{lem:time-bounded} relativizes in the following way.	
	
\begin{Lemma}\label{lem:relativized-time-bounded}
Let $A \in \sequences$. For any given computable time bound~$t$, $\mathbf{m}^{A,t}$ is a semimeasure and $\mathbf{m}^{A,t} \leq_{tt} A$. Conversely, if $m$ is a semimeasure $tt$-below $A$, there exists a computable time bound~$t$ such that $m \leq^\times \mathbf{m}^{A,t}$. In particular, for any given time bound~$t$, $2^{-K^{A,t}}$ is a semimeasure that is $tt$-below $A$ , hence there exists a computable time bound $t'$ such that $2^{-K^{A,t}} \leq^\times \mathbf{m}^{A,t'}$ (or equivalently, $-\log m^{A,t'} \leq^+ K^{A,t}$). 
\end{Lemma}

\begin{proof}
The first part of the lemma is immediate, provided one recalls that the relativized version $\mathbf{m}^A$ of the universal lower semicomputable semimeasure can (and should!) be defined uniformly, that is there is a two-place function $\Phi: \sequences \times \N \rightarrow \mathbb{R}^+$  such that 

\begin{itemize}
\item[(i)] The set $\{(A,n,q) \in \sequences \times \N \times \Q \mid \Phi(A,n)>q\}$ is $\Sigma^0_1$
\item[(ii)] For every $A$, $n \mapsto \Phi(A,n)$ is a semi-measure  
\item[(iii)] For any other function $\Psi$ with property (i), there exists a constant $c>0$ such that for each~$B$, if $n \mapsto \Psi(B,n)$ is a semimeasure, then $\Phi(B,n)>c \cdot \Psi(B,n)$ for all~$n$. 
\end{itemize}

(see for example~\cite{Gacs-notes} for a proof of the existence of such a $\Phi$) and thus one may \emph{define} $\mathbf{m}^A(n)$ to be $\Phi(A,n)$ for all~$A$ and $n$ (and $\mathbf{m}^{A,t}$ its time-bounded version, canonically defined by adding a time bound on~$\Phi$). With this definition, we do get $\mathbf{m}^{A,t} \leq_{tt} A$.

Conversely, suppose $m$ is a semimeasure which is $tt$-below $A$. Let $\Psi$ be the total functional such that $m = \Psi^A$. By item (iii) above, let $c$ be a constant such that $\mathbf{m}^B(n)>c \cdot \Psi^B(n)$ for all~$n$ whenever $\Psi^B$ is a semimeasure. This means that for all~$n$ the $\Pi^0_1$ class 
\[
\left\{B \mid \Psi^B \text{ is a semimeasure and } c \cdot \Psi^B(n) > \mathbf{m}^B(n)\right\} 
\]
must be empty, hence by effective compactness of $\sequences$ one can effectively compute some $t(n)$ such that 
\[
\left \{B \mid \Psi^B \text{ is a semimeasure and } c \cdot \Psi^B(n) > \mathbf{m}^B(n)[t(n)] \right\} =\emptyset 
\]
Since $\Psi^A=m$ is a semimeasure, it follows that $c \cdot m \leq  \mathbf{m}^{A,t}$. 
\end{proof}

	It is now easy to see that the following properties holds for relativized depth, by simply relativizing the proofs of the corresponding properties of unrelativized depth.
	
	\begin{Proposition} \label{basic_properties}
		Let $A \in \sequences$
		\begin{enumerate}[(i)]
			\item (Relativized SGL) Let $X \leq_{tt} Y$ and $X$ be $A$-deep. Then $Y$ is $A$-deep.
			\item Every $A$-ML-random set is $A$-shallow.
			\item Every $A$-$tt$-computable set is $A$-shallow.
			\item $A^{\prime}$ is $A$-deep.
		\end{enumerate}
	\end{Proposition}

	Next theorem shows how relativized depth is preserved when considering different oracles.
	
	\begin{Theorem} \label{preservingdepth}
		Let $A, B \in \sequences$ such that $A \leq_ {tt} B$ and $A \equiv_T B$. Then every $B$-deep set is also $A$-deep.
	\end{Theorem}
	\begin{proof}
		Let $X \in \sequences$ be $A$-shallow. Hence, there is a computable time-bound $t$ such that 
		$$\left( \io n \right)\left[K^{A,t}(X \upharpoonright n) =^+ K^A(X \upharpoonright n)\right].$$
		
		Since $A \equiv_T B$, by Lemma \ref{K_lemma} (i), we get
		$$\left( \forall n \right)\left[K^A(X \upharpoonright n) =^+ K^B(X \upharpoonright n)\right].$$
		Moreover, by Lemma \ref{K_lemma} (ii), since $A \leq_{tt} B$, we also have that 
		$$\left( \forall n \right)\left[ K^{B,t'}(X \upharpoonright n) \leq^+ K^{A,t}(X \upharpoonright n) \right].$$
		
		Hence, we get 
		$$\left( \io n \right)\left[K^{B,t'}(X \upharpoonright n) \leq^+ K^{A,t}(X \upharpoonright n) =^+ K^A(X \upharpoonright n) =^+ K^B(X \upharpoonright n)\right],$$
		so that $X$ is $B$-shallow.	
	\end{proof}

	Notice that usually, in computability theory, relativizing a class $\mathcal{C}$ defines a class $\mathcal{C}^A$ such that either $\mathcal{C}^A \subseteq \mathcal{C}$ (e.g., when $\mathcal{C}$ is the class of ML-random sets) or $\mathcal{C} \subseteq \mathcal{C}^A$ (e.g., when $\mathcal{C}$ is the class of computable sets), for all oracles $A$. 
	
	Being defined in terms of two quantities which decrease mutually independently when an oracle is applied, this is not the case of the classes of deep and shallow sets. A priori, for an oracle $A$, we have four possible different scenarios: 
	\begin{enumerate}
		\item The classes of $A$-deep sets and deep sets are incomparable, meaning that there are both shallow but $A$-deep sets and deep but $A$-shallow sets. In section \ref{hp-depth}, we show that $\hp$ is an example for this scenario.
		\item All deep sets remain deep relative to $A$, but there are shallow sets which look deep relative to $A$, so that depth implies $A$-depth, but the reverse implication fails: Section \ref{MLR} is devoted to show that this is the case of ML-random oracles.
		\item All shallow sets remain shallow relative to $A$, but there are deep sets which look shallow relative to $A$, hence $A$-depth implies depth, but the converse does not hold. 
		\item The class of $A$-deep sets and the one of deep sets coincide.
	\end{enumerate}
	
	In Section \ref{K-trivial}, we observe that $K$-trivial oracles are examples of either scenario~3 or 4. However, while the latter obviously applies to all computable sets, we do not know which of them holds for uncomputable K-trivial oracles. Notice that, a priori, different $K$-trivial oracles may give different answers.

	
\section{$\hp$-depth} \label{hp-depth}
	In this section, we build a $\Delta_2^0$ set which is $\hp$-deep but ML-random, hence shallow. This construction shows, in particular:
	\begin{enumerate}[(i)]
		\item There are oracles $A$ such that some $A$-computable sets are deep. (Notice that, by Proposition \ref{basic_properties} (iii), this is not the case of $A$-$tt$-computable sets, namely sets computable with oracle $A$ within some computable time-bound). 
		\item There are oracles $A$ such that the class of $A$-deep sets is incomparable with the corresponding unrelativized class: indeed, every c.e.~deep set (including $\hp$) is clearly $\hp$-shallow (as $tt$-below $\hp$), while the set we construct in next theorem is $\hp$-deep but shallow.
	\end{enumerate}

	The construction is based on the following technical lemma (stated in the following form in \cite[Remark 3.1]{MM}, see also \cite{G}).
	
	\begin{Lemma}[Space Lemma]\label{SL}

For a rational~$\delta>1$ and positive integer~$k$, let $l(\delta,k) = \lceil \frac{k+1}{1-1/\delta} \rceil$. 
		For every martingale $d$ and $\sigma \in \strings$, 
		$$\left| \left\lbrace \tau \in \mathbf{2}^{l(\delta,k)}: \ d(\sigma\tau) < \delta d(\sigma) \right\rbrace \right| \geq k.$$
	\end{Lemma}

Let us now prove the main theorem of this section. 

	\begin{Theorem} \label{XMLRandHPdeep}
		There exists a set $X \in \Delta^0_2$ such that $X$ is ML-random and $X$ is $\hp$-deep.
	\end{Theorem}

\begin{proof}
Let $\mathbf{d}$ be a universal lower-semicomputable martingale and let $T: \mathbb{N} \rightarrow \mathbb{N}$ be a $\Delta^0_2$ function such that for any computable time bound~$t$ we have $\foralmostall n~t(n) \leq T(n)$. For $n>0$ let  $\delta_n = 1+1/n^2$.  

The construction of~$X$ will be done by a finite extension method, that is, we will build~$X$ as the limit of an increasing (w.r.t. the lexicographic order) sequence of strings~$(\tau_n)$, the construction being effective in $\emptyset'$.  We set $\sigma_0$ to be the empty string~$\lambda$. Having built $\sigma_n$, $\sigma_{n+1}$ is chosen as follows. By the Space Lemma, the set 
\[
A_{n+1} = \left\lbrace \tau \in \mathbf{2}^{l(\delta_{n+1},2^{n+1})}: \ \mathbf{d}(\sigma_n\tau) < \delta_{n+1} \cdot \mathbf{d}(\sigma) \right\rbrace
\]
has at least $2^{n+1}$ elements. Thus, there must exist some $\tau \in A_{n+1}$ such that $K^{\emptyset'}(\sigma_n \tau) > n$ and a fortiori there must exist some $\tau$ in $A_{n+1}$ such that $K^{\emptyset',F}(\sigma_n \tau) > n$. Since the latter is a $\emptyset'$-decidable property and $A_{n+1}$ is $\emptyset'$-c.e, one can $\emptyset'$-effectively find such a $\tau$ and set $\sigma_{n+1} = \sigma_n \tau$.

We verify that this construction works via a series of claims. \\

\emph{Claim 1: $X$ is $\Delta^0_2$}. \\

This is clear because the whole construction is effective relative to $\emptyset'$. \\

\emph{Claim 2: For all~$n$, $\mathbf{d}(\sigma_n) \leq \mathbf{d}(\lambda) \cdot \prod_{i=1}^n (1+ 1/i^2)$. Thus the sequence $\mathbf{d}(\sigma_n)$ is bounded, and thus $X$ is Martin-L\"of random }.\\

The inequality is proved by induction. It is obvious for $n=0$ and if we have it for some $n$, then by choice of $\tau$ and $\sigma_{n+1}$, we have 
\[
\mathbf{d}(\sigma_{n+1}) \leq \delta_{n+1} \mathbf{d}(\sigma_n) \leq (1+1/(n+1)^2) \mathbf{d}(\sigma_n)\leq (1+1/(n+1)^2) \cdot \mathbf{d}(\lambda) \cdot \prod_{i=1}^n (1+ 1/i^2) 
\] 
 which finishes the induction. Thus we have, for all~$n$, $\mathbf{d}(\sigma_n)$ is bounded by $\mathbf{d}(\lambda) \cdot \prod_{i=1}^\infty (1+1/i^2)$ which is a finite real. Since all $\sigma_n$ are initial segments of~ $X$, this means that $\liminf_k \mathbf{d}(X \upharpoonright k)$ is finite, hence $X$ is Martin-L\"of random. \\

\emph{Claim 3: $|\sigma_n|=n^2/2 + O(n \log n)$}.\\

Indeed we have by construction 
\[
|\sigma_{n+1}|=|\sigma_n|+l(\delta_{n+1},2^{n+1}) = |\sigma_n| + \left \lceil \log \left( \frac{2^{n+1}+1}{1-\frac{1}{1+1/(n+1)^2}} \right) \right\rceil = |\sigma_n| + n + O(\log n)
\]
and thus, by summation, $|\sigma_n| = \frac{n^2}{2} + O(n \log n)$. \\

\emph{Claim 4: for any computable time bound~$t$, $K^{\emptyset',t}(\sigma_n) \geq n$ while $K^{\emptyset'}(\sigma_n) = O(\log n)$}.\\

Let $t$ be a computable time bound. For almost all~$k$, $t(k) \leq F(k)$ hence for almost all~$\tau$, $K^{\emptyset',t}(\tau) \geq K^{\emptyset',F}(\tau)$. In particular, for almost all~$n$, $K^{\emptyset',t}(\sigma_n) \geq K^{\emptyset',F}(\sigma_n) \geq n$ (the last inequality is by choice of $\sigma_{n+1}$ in the construction). On the other hand, since the sequence of $\sigma_n$ is $\emptyset'$-computable, we have $K^{\emptyset'}(\sigma_n) =^+ K^{\emptyset'}(n) = O(\log n)$. \\

\emph{Claim 5: $X$ is $\emptyset'$-deep}.\\

Let $t$ be a computable time bound. For any~$k$, by construction of $X$, there is an $n$ such that 
\[
\sigma_n \preceq (X \upharpoonright k) \preceq \sigma_{n+1}
\]
We can recover $\sigma_n$ from $|\sigma_n|$ and $X \upharpoonright k$ by just truncating the latter to its first~$|\sigma_n|$ bits. Since $|\sigma_n|$ is computable in~$n$, this means that if~$t'$ is sufficiently fast-growing,
\[
K^{\emptyset',t} (X \upharpoonright k) \geq K^{\emptyset',t'} (\sigma_n) - K^{\emptyset',t}(n) \geq n - O(\log n)
\]
(the last inequality following from Claim 4). 

By the same reasoning, we can recover $X \upharpoonright k$ from $\sigma_{n+1}$ and $k$, hence 
\[
K^{\emptyset'}(X \upharpoonright k) \leq K^{\emptyset'}(\sigma_{n+1}) + K^{\emptyset'}(k) \leq O(\log n) + O(\log k) 
\]
Moreover, since $|\sigma_n| \leq k \leq |\sigma_{n+1}|$, by Claim 3, $k \sim n^2/2$, meaning that one can replace $O(\log k)$ by $O(\log n)$ in the above expression. Putting both inequalities together, we get 
\[
K^{\emptyset',t} (X \upharpoonright k)- K^{\emptyset'} (X \upharpoonright k) \geq n - O(\log n) \geq \sqrt{2k} - o(\sqrt{k})\]
which finally shows that 
\[
K^{\emptyset',t} (X \upharpoonright k) - K^{\emptyset'} (X \upharpoonright k) \rightarrow \infty
\]
This finishes the proof. 
\end{proof}

	\section{Depth relative to ML-random oracles} \label{MLR}
	The main goal of this section is to prove that depth is strictly implied by depth relative to any ML-random oracle. We will first prove that ML-random sets are non-trivial examples of oracles for which all deep sets remain deep relative to them. We successively show that shallowness is a notion preserved by almost every oracle, namely that if a set is shallow, then it is shallow relatively to a class of oracles of measure 1. Despite this fact, every ML-random set ``adds'' deep sets: for every ML-random oracle we can find a shallow set which is deep relatively to that oracle. Interestingly, as a consequence of the existence of such sets, we can give a quite short and easy proof of the fact, originally proved by Barmpalias, Lewis and Ng~\cite{BLN}, that every PA-complete degree is the join of two ML-random degrees.
	
	\subsection{Deep sets remain deep relative to ML-random oracles} 
	In order to prove that no deep set can be shallow relative to ML-random oracles, we recall another characterization of ML-randomness.
	
	\begin{Definition}
		
		$\Psi: \sequences \rightarrow [0, \infty]$ is an integral test if
		\begin{itemize}
			\item $\Psi$ is \emph{lower-semicomputable}, i.e.~it is the supremum of a computable sequence of computable functions $\Psi_n: \sequences \rightarrow [0, \infty)$, and 
			\item $\int_{X \in \sequences} \Psi(X)\, d\mu \leq 1 $.
		\end{itemize}
		
	\end{Definition}
	
	Integral tests characterize ML-randomness, in the following sense. (For a proof of the statement below, see \cite{BGKRS}).
	
	\begin{Proposition}
		$X$ is not ML-random if and only if there is an integral test~$\Psi$ such that $\Psi(X) = \infty$.
	\end{Proposition}

	We can now turn to the main result.
	\begin{Theorem} \label{MLRlowfordepth}
	Let $A \in \sequences$ be ML-random. Then every deep set is $A$-deep.
	\end{Theorem}
	
	\begin{proof}
		We prove, in particular, that, if $A \in \sequences$ is ML-random, then for every computable time-bound $t$ there exists a computable time-bound $t'$ such that 
		$$\left(\forall \sigma \right) \left[ K^{t'}(\sigma) - K(\sigma) \le^+ K^{A,t}(\sigma) - K^A(\sigma) \right]$$
		
which will immediately imply the theorem. \\
		
		For any string $\sigma$ and time-bound $t$, let $\mathbf{m}(\sigma) = 2^{-K(\sigma)}$, $\mathbf{m}^t(\sigma) = 2^{-K^t(\sigma)}$, $\mathbf{m}^A(\sigma) = 2^{-K^A(\sigma)}$ and $\mathbf{m}^{A,t}(\sigma) = 2^{-K^{A,t}(\sigma)}$.
		
		For every $A \in \sequences$ and every computable time-bound $t$, it clearly holds that 
		$$\sum_{\sigma \in \strings} \mathbf{m}^{A,t}(\sigma) \leq \sum_{\sigma \in \strings} \mathbf{m}^A(\sigma) \leq 1.$$
		Hence, 
		$$\int_{A \in \sequences} \sum_{\sigma \in \strings} \mathbf{m}^{A,t}(\sigma) \,d\mu \leq 1$$
		and, by interchanging the sum with the integral, 
		$$\sum_{\sigma \in \strings} \left( \int_{A\in \sequences} \mathbf{m}^{A,t}(\sigma) \,d\mu \right) \leq 1,$$
		which means that the map $\sigma \mapsto \int_{A \in \sequences} \mathbf{m}^{A,t} \,d\mu$ is a discrete semimeasure. Moreover, it is computable. Thus, by Lemma~\ref{lem:time-bounded},
		$$\int_{A\in \sequences} \mathbf{m}^{A,t}(\sigma) \,d\mu \leq c \cdot \mathbf{m}^{t'}(\sigma),$$
		for some computable time-bound $t'$.

		Now, consider the map $\Psi: \sequences \rightarrow [0, \infty]$ given by 
		$$\Psi(A)= \sum_{\sigma \in \strings} \frac{\mathbf{m}^{A,t}(\sigma) \mathbf{m}(\sigma)}{c \cdot \mathbf{m}^{t'}(\sigma)}.$$
		Then $\Psi$ is lower-semicomputable and, by the above computations, $$\int_{A \in \sequences} \Psi(A) \,d\mu \leq 1,$$
		that is, $\Psi$ is an integral test. Let $A$ be a ML-random set, so that $\Psi(A) < k$ for some $k \in \mathbb{N}$. By definition of~$\Psi$ this means that 
		\[
		\sum_{\sigma \in \strings} \frac{\mathbf{m}^{A,t}(\sigma) \mathbf{m}(\sigma)}{c \cdot \mathbf{m}^{t'}(\sigma)} < k
		\]
		and thus 
		$\sigma \mapsto \frac{\mathbf{m}^{A,t}(\sigma) \mathbf{m}(\sigma)}{kc \cdot \mathbf{m}^{t'}(\sigma)}$ is an $A$-lss. It follows that for every $\sigma \in \strings$, 
		$$\frac{\mathbf{m}^{A,t}(\sigma) \mathbf{m}(\sigma)}{\mathbf{m}^{t'}(\sigma)} \leq^{\times} \mathbf{m}^A(\sigma),$$
		which, by taking the logarithm of both sides, implies 
		$$K^{t'}(\sigma) - K(\sigma) \leq^+ K^{A,t}(\sigma) - K^{A}(\sigma),$$	
		
		which is what we wanted.	
	\end{proof}

	As a consequence, we get our first example of an oracle making the class of deep sets strictly smaller.

	\begin{Corollary}\label{omega}
		The class of $\Omega$-deep strictly contains the class of deep sets.
	\end{Corollary}

	\begin{proof}
		Since $\Omega$ is ML-random, every deep set is $\Omega$-deep. It is then enough to find a set $X$ which is shallow but $\Omega$-deep. Let $X$ be the set built in Theorem~\ref{XMLRandHPdeep}. Recall that $\Omega \leq_{tt} \hp$ and $\Omega \equiv_T \hp$, hence, by Theorem \ref{preservingdepth}, $X$ is $\Omega$-deep.
	\end{proof}

	\subsection{A shallow sets which is deep relative to a ML-random oracle}
	This section is mainly devoted to showing that what we observe for $\Omega$ in Corollary \ref{omega} holds, in fact, for every ML-random oracle. To show this fact, we need to recall some terminology. 
	
	\begin{Definition}
		$f: \mathbb{N} \rightarrow \mathbb{N}$ is a \emph{Solovay function} if 
		$$(\forall n)[f(n) \le^+ K(n)] \quad \text{and} \quad \left(\io n \right)[K(n) \le^+ f(n)] .$$
	\end{Definition}
	As shown in \cite{HKM}, for any superlinear time-bound $t$, the time-bounded Kolmogorov complexity $K^{t}$ is a computable Solovay function. 
	
	We first observe that any shallow set remains shallow relative to almost every oracle, where ``almost every'' is meant in the measure-theoretic sense. Namely, we have the following result. 
	
	\begin{Theorem}
		If $X \in \sequences$ is shallow, then $$\mu\left( \left\lbrace A: \ X \ \text{is} \ A\text{-shallow} \right\rbrace \right)=1.$$
		
		In particular, $X$ is $A$-shallow for every $A$ which is $2$-random relative to $X$ (i.e.~$A \notin \cap_n \mathcal{U}_n^X$, where $(\mathcal{U}_n^X)_{n \in \mathbb{N}}$ is any sequence of uniformly $\Sigma_2^X$ classes, with $\mu(\mathcal{U}_n^X) \le 2^{-n}$ for all $n$).
	\end{Theorem}

	\begin{proof}
		Since $X$ is shallow, there are a computable time-bound $t$ and $c \in \mathbb{N}$ such that
		$$\left(\io n \right)\left[\frac{\mathbf{m}(X \upharpoonright n)}{\mathbf{m}^t(X \upharpoonright n)} < c \right].$$
		
		Recall that the map 
		$$\sigma \mapsto \int_{A\in \sequences} \mathbf{m}^A(\sigma)\, d\mu$$
		is a lower-semicomputable discrete semimeasure. Hence, by the universality of $\mathbf{m}$, we get that $\mathbf{m}(\sigma) \ge^{\times} \int_{A \in \sequences} \mathbf{m}^A(\sigma) \, d\mu$ for all $\sigma \in \strings$. On the other hand, there exists a constant~$c'$ such that $c' \mathbf{m}^A > \mathbf{m}$ for any~$A$, hence $\int_{\sequences} \mathbf{m}^A(\sigma) \ge^{\times} \mathbf{m}(\sigma)$. Putting the two together, 
		\begin{equation} \label{average}
		 \mathbf{m}(\sigma) =^{\times} \int_{A \in \sequences} \mathbf{m}^A(\sigma) \,d\mu.
		\end{equation}
		
		For any $k \in \mathbb{N}$ let
		$$\mathcal{L}_k = \left\lbrace Y : \ \left(\foralmostall n \right)\left[\frac{\mathbf{m}^Y(X \upharpoonright n)}{\mathbf{m}^t (X \upharpoonright n)} \ge k \right] \right\rbrace.$$
		Now, suppose that $X$ is $A$-deep, so that, in particular, 
		$$\left(\forall k \foralmostall n \right)\left[\frac{\mathbf{m}^A(X \upharpoonright n)}{\mathbf{m}^t(X \upharpoonright n)} \ge k \right],$$
		meaning that $A \in \mathcal{L}_k$ for every $k$. It is then enough to show that 
		$$\mu\left(\bigcap_{k \in \mathbb{N}} \mathcal{L}_k \right)=0,$$
		namely that
		$$\lim\limits_{k \rightarrow \infty}\mu(\mathcal{L}_k) = 0.$$
		Observe that 
		$$\mathcal{L}_k = \liminf_{n \rightarrow \infty} \mathcal{U}_n^k,$$
		where
		$$\mathcal{U}_n^k =  \left\lbrace Y : \ \left[\frac{\mathbf{m}^Y(X \upharpoonright n)}{\mathbf{m}^t (X \upharpoonright n)} \ge k \right] \right\rbrace.$$ 
		We claim that 
		$$\left(\io n \right)\left[ \mu\left(\mathcal{U}_n^k\right) \le^{\times} \frac{1}{k} \right].$$
		Indeed, since $K^t$ is a Solovay function, 
		$$\left( \io n \right)\left[\mathbf{m}^t (X \upharpoonright n) \ge^{\times} \mathbf{m}(X \upharpoonright n) \right].$$
		Then, for any such $n$, if $Y \in \mathcal{U}_n^k$, then
		$$\mathbf{m}^Y(X \upharpoonright n) \ge^{\times} k\cdot \mathbf{m}^t(X \upharpoonright n) \ge^{\times} k \cdot \mathbf{m}(X \upharpoonright n) \ge^{\times} k\cdot \int_{A \in \sequences} \mathbf{m}^A(X \upharpoonright n) \,d\mu,$$
		where the last inequality follows from (\ref{average}). Our claim follows then by Marokov's inequality. Therefore, we get 
		$$\mu(\mathcal{L}_k) = \mu (\liminf_{n \rightarrow \infty} \mathcal{U}_n^k) \leq \liminf_{n \rightarrow \infty} \mu (\mathcal{U}_n^k) \leq^{\times} \frac{1}{k},$$
		where the first inequality follows by Fatou's Lemma, and hence 
		$$\lim_{k \rightarrow \infty} \mu(\mathcal{L}_k) \le^{\times} \lim_{k \rightarrow \infty} \frac{1}{k}=0.$$
		
		Finally, observe that $\mathcal{L}_k$ is the limit inferior of a uniformly c.e.~sequence of $\Sigma_1^{0,X}$ classes with bounded measure. Hence, as proven in \cite{B}, each $\mathcal{L}_k$ is contained in a $\Sigma_2^{0, X}$ class of bounded measure. Hence, the classes $\mathcal{L}_k$ form a test for $2$-randomness relative to $X$, so that $X$ must be $A$-shallow for any $A$ which is $2$-random relative to $X$.
	\end{proof}
 	
 	It is then natural to ask whether for a ``very random'' oracle every set must be shallow. In other words, whether there is a randomness notion sufficiently strong to make any such oracle totally useless in organizing information. We answer this question in the negative. Indeed, we show that every ML-random oracle makes some shallow set deep. Intuitively, the proof of this fact is similar to the one-time pad protocol in cryptography: we can ``mix'' together some important piece of information $x$ with some random string $a$ we know, so that the output $x \triangle a$ still looks important for us (as we can distinguish the added random noise $a$), while looking random to the others. 
 	
 	In order to prove formally our claim, we recall some well-known notions and facts. 
 	
 	\begin{Definition} Let $(\phi_e)$ be an effective enumeration of all partial computable functions from $\mathbb{N}$ to $\mathbb{N}$. A total function $f$ is \emph{diagonally non-computable} (DNC) if $f(e) \not=\phi_e(e)$ whenever $\phi_e(e)$ is defined. We say that $f$ is $\text{DNC}_2$ if it is DNC and its range is $\{0,1\}$. The set $\text{DNC}_2$ of such functions is a $\Pi^0_1$ class.  A set $X$ is said to be\emph{PA-complete} if it computes some member of $\text{DNC}_2$. 
	
 	\end{Definition}
	
(The terminology `PA-complete' comes from the equivalent definition where one replaces the class $\text{DNC}_2$ by the class of complete coherent extensions of Peano Arithmetic). 
 
 A well-known property of the class $\text{DNC}_2$ is that it is \emph{universal} (a.k.a.\ \emph{Medvedev-complete}) among $\Pi^0_1$ classes, that is, for every non-empty $\Pi^0_1$ class~$\mathcal{C}$, there exists a total functional $\Phi$ such that $\Phi^X \in \mathcal{C}$ for any $X \in \text{DNC}_2$. Combined with a result of Moser and Stephan, this yields the following lemma. 
 
 \begin{Lemma}\label{lem:pa-deep}
 Every $X \in \text{DNC}_2$ is deep. 
 \end{Lemma}
 
 \begin{proof}
 Indeed, in~\cite{MS}, it is shown that there exists a non-empty $\Pi_1^0$ class $\mathcal{C}$ in which every member is a deep set. Since $\text{DNC}_2$ is universal, every $X \in \text{DNC}_2$ tt-computes some member of $\mathcal{C}$ (via the same functional). Since depth is tt-closed upwards, this proves our result. 
 \end{proof}

	Thus we can use well-known basis theorems to obtain deep sets with some desired properties. In particular, we will use the following fact (for a proof, see~\cite{DHMN}). 
 	
 	\begin{Proposition}[Randomness Basis Theorem] \label{RBT} Let $R$ be a ML-random set. Every non-empty $\Pi_1^0$ class contains an element $X$ such that $R$ is $X$-ML-random.
 	\end{Proposition}
 
 	We are now ready to prove the main result of this section.
	\begin{Theorem} \label{MLRnotlowfors}
 		If a set $A$ is ML-random, then there exists a shallow set~$X$ which is $A$-deep.
 	\end{Theorem}
 
 	\begin{proof}
 	 	Let $A$ be a ML-random set. By Proposition \ref{RBT}, there exists a deep set~$X$ such that $A$ is ML-random relative to $X$. On the other hand, by Theorem \ref{MLRlowfordepth}, $X$ is also $A$-deep. We then consider the set 
 	 	$$A \triangle X = (A \smallsetminus X) \cup (X \smallsetminus A).$$
 	 	
 	 	We first notice that, since $A$ is ML-random relative to $X$, $A \triangle X$ is also ML-random relative to $X$ (as $K^X((A \triangle X) \upharpoonright n) =^+ K^X(A \upharpoonright n) \ge^+ n$) and hence shallow.
 	 	
 	 	On the other hand, since $X$ is $A$-deep and $X \leq_{tt} (A \triangle X) \oplus A$, by Proposition~\ref{basic_properties} (relativized SGL) $(A\triangle X) \oplus A$ must also be $A$-deep, which in turn implies that $A\triangle X$ is $A$-deep (as neither $K^A$ nor its time-bounded variants are affected by the join of $A$). 
 	\end{proof} 
 
 	\subsection{A digression on a result about PA-complete degrees}
 	
 	In this section we observe that the existence of sets as in the proof of Theorem~\ref{MLRnotlowfors} can improve upon (and give a simpler proof of) a theorem of Barmpalias, Lewis and Ng (\cite{BLN}), who proved that for every PA-complete $A$, there exist two Martin-L\"of random $X,Y$ such that $A \equiv_T X \oplus Y$. We will prove the following. 
 	
 	\begin{Theorem}\label{thm:bln-improved}
 		Let $A \in \text{DNC}_2$. Then there exist two Martin-L\"of random $X,Y$ such that $A = X \triangle Y$ (in particular $A \equiv_{tt} X \oplus Y$). 
 	\end{Theorem}

In order to prove Theorem~\ref{thm:bln-improved}, we need the following property of universal $\Pi^0_1$ classes.
 	
 	\begin{Lemma}  \label{lem:universal-pi01}
	Let $\mathcal{C}$ be a non-empty universal $\Pi_1^0$ class. 
	\begin{itemize}
 		\item[(i)] For any $A \in \text{DNC}_2$, there is a set $B \in \mathcal{C}$ such that $A \equiv_{tt} B$
		\item[(ii)] For any $A$ which Turing-computes some $\text{DNC}_2$ function, there exists some $B \in \mathcal{C}$ such that $A \equiv_{T} B$. 
	\end{itemize}
 	\end{Lemma}
 
 	\begin{proof}

$(i)$~	Let $\Phi$ witness the universality of $\mathcal{C}$, that is, $\Phi^X \in \text{DNC}_2$ whenever $X \in \mathcal{C}$. 
	
 		The construction of $B$ is done via an $A$-effective forcing argument: we define a sequence $\mathcal{C}_0 \supseteq \mathcal{C}_1 \supseteq \dots$ of $\Pi_1^0$ classes such that the sequence of indices for the $\mathcal{C}_i$ is $A$-computable, each of these classes is non-empty and that their intersection $\cap_s \mathcal{C}_s$ is a singleton which will be our~$B$.
 		
		To ensure that $\cap_s \mathcal{C}_s$ will be singleton, we will have at each step a string $\sigma_s$ such that $|\sigma_s|=s$ and $\mathcal{C}_s \subseteq [\sigma_s]$. 
		
 		Let $\mathcal{C}_0=\mathcal{C}$ and $\sigma_0=\lambda$. Inductively, assume that we have already built non-empty $\Pi_1^0$ classes $\mathcal{C}_0 \supseteq \dots \supseteq \mathcal{C}_s$ and strings $\sigma_0 \preceq \sigma_1 \preceq \ldots \preceq \sigma_s$ such that $\mathcal{C}_i \subseteq [\sigma_i]$ for all~$i \leq s$. 
	
	Using an index for $\mathcal{C}_s$, one can compute some  $n_{s}$ such that $\phi_{n_s}(n_s)$ returns the first~$i \in \{0,1\}$ found such that if $\mathcal{C}_s \cap [\sigma_s i] = \emptyset$ (and stays undefined if no such~$i$ is found). Since $A \in \text{DNC}_2$, we have $\phi_{n_s}(n_s) \not= A(n_s)$. Let thus $\sigma_{s+1} = \sigma_s A(n_s)$, so that $\mathcal{C}_s \cap [\sigma_{s+1}]$ is non-empty. 
	
Now, by the Recursion Theorem, let $m_s$ be an index such that $\phi_{m_s}$ waits for some~$i \in \{0,1\}$ to be such that 
\[
\left(\forall X \in \mathcal{C}_s \cap [\sigma_{s+1}]\right) \, \Phi^X(m_s) = i
\]
(since $\mathcal{C}_s \cap [\sigma_{s+1}]$ is a $\Pi^0_1$ class, this property is a c.e.\ event) and if such an~$i$ is found returns $\phi_{m_s}(m_s)=1-i$ (staying undefined otherwise).

Since the image of $\mathcal{C}_s \cap [\sigma_{s+1}]$ is contained in $\text{DNC}_2$, one must have $\Phi^X(m_s) \not= \phi_{m_s}(m_s)$ for any $X \in \mathcal{C}_s$. This means that in fact $\phi_{m_s}(m_s)$ must be undefined, which in turns implies that for all~$i \in \{0,1\}$, 
  
\[
\mathcal{C}_s \cap [\sigma_{s+1}] \cap \{X \, :  \Phi^X(m_s) = i\} \not=\emptyset.
\] 

Now define
\[
\mathcal{C}_{s+1} = \mathcal{C}_s \cap [\sigma_{s+1}] \cap \{X \, : \Phi^X(m_s) = A(s)\}
\]
  
Finally, define~$B$ to be the unique element of $\bigcap_s \mathcal{C}_s$ (or equivalently the limit of the $\sigma_s$). We claim that $B$ is as wanted. 

To see that $A$ computes $B$, observe that the whole construction is $A$-tt-effective. If we have already $A$-tt-computed an index for $\mathcal{C}_s$ and $\sigma_s$, $n_s$ can be effectively tt-computed and $\sigma_{s+1}=\sigma_s A(n_s)$ can be computed from $A$. Likewise, since the Recursion Theorem is effective, $m_s$ can then be computed, and an index for $\mathcal{C}_{s+1}$ tt-effectively obtained from $m_s$ and $A$. 

But the whole construction (indices for the $\mathcal{C}_s$ and strings $\sigma_s$) can also be tt-recovered from $B$. Indeed, $\sigma_{s}$ is the prefix of $B$ of length~$s$, thus $A(n_s)$ can be tt-computed from~$B$ for all~$s$, and to compute an index for $\mathcal{C}_{s+1}$ knowing one for $\mathcal{C}_{s}$ we only need to know $m_s$ (which can be tt-computed from the index of $\mathcal{C}_{s}$ and $\sigma_{s+1}$) and $A(s)$ which is just $\Phi^B(m_s)$ (and recall that $\Phi$ is a tt-functional). This implies that $B$ tt-computes the sequence of $m_s$ and thus $A(s)$ for each~$s$ since $A(s) = \Phi^B(m_s)$. \\

$(ii)$~The proof of this part is almost identical. If $A \geq_T F$ for some $\text{DNC}_2$ function, make the same construction only replacing $A(n_s)$ by $F(n_s)$. Then $A$ computes~$B$ (not necessarily in a tt-way since it needs to compute~$F(n_s)$ for all~$s$ to do so) and $B$ tt-computes~$A$.  
\end{proof}

%
 
 We can now prove Theorem \ref{thm:bln-improved}. \\
 
 \begin{proof}[Proof of Theorem~\ref{thm:bln-improved}]

 	For any $k \in \mathbb{N}$, let
 	$$\text{MLR}_k = \{X :\ (\forall n)K(X \upharpoonright n) \ge n-k\}.$$
	
	(which is a $\Pi^0_1$ class). By the characterization of Martin-L\"of randomness in terms of prefix-free Kolmogorov complexity, a set $X$ is Martin-L\"of random if and only if $X \in \text{MLR}_k$ for some large enough~$k$. 
	
 	In the proof of Theorem \ref{MLRnotlowfors} we saw  that for $k$ large enough, the $\Pi_1^0$ class 
	\[
\mathcal{C} = \{F \oplus X \oplus Y \, :  F \in \text{DNC}_2, \ X,Y \in \text{MLR}_k, \ F=X \triangle Y \} 
	\]
	
 	is non-empty. Moreover, the class $\mathcal{C}$ is universal , as witnessed by the first projection.  Let $A \in \text{DNC}_2$. By the previous lemma, $A \equiv_{tt} F \oplus X \oplus Y$ for some $F \oplus X \oplus Y  \in \mathcal{C}$. Clearly, $X \oplus Y \le_{tt}  F \oplus X \oplus Y$. On the other hand, since $F = X \triangle Y$, we also have that $F \le_{tt} X \oplus Y$ and hence $A \equiv_{tt} F \oplus X \oplus Y \equiv_{tt} X \oplus Y$. 
 	 
\end{proof}	 

\begin{Corollary}[Barmpalias, Lewis, Ng]
If $A$ has PA-complete Turing degree, there are two Martin-L\"of random sets $X,Y$ such that $A \equiv_T X \oplus Y$. 
\end{Corollary}

\begin{proof}
Since $A$ has PA-complete Turing degree, by Lemma~\ref{lem:universal-pi01} (where $\mathcal{C}$ is taken to be the class $\text{DNC}_2$), $A$ is Turing equivalent to some $\tilde{A} \in \text{DNC}_2$. By the previous theorem, $\tilde{A}$ is in turn tt-equivalent (hence Turing equivalent) to the join of two Martin-L\"of random sets. 
\end{proof}
	 
	\section{An open question about K-trivial oracles} \label{K-trivial}
	
	The intuitive notion of \emph{being far from being ML-random} turned out to be a central notion in algorithmic randomness. This is formally expressed by the following definition. 
	
	\begin{Definition}
		A set $X \in \sequences$ is \emph{K-trivial} if 
		$$(\forall n)K(X \upharpoonright n) \le^+ K(n).$$
	\end{Definition}
	
	In other words, describing a prefix of $X$ is at most as hard as describing its length. Another lowness property related to prefix-free complexity is when an oracle does not help in further compressing strings. 
	
	\begin{Definition}
		A set $X \in \sequences$ is \emph{low for K} if 
		$$(\forall \sigma)\left[K(\sigma) \le^+ K^A(\sigma)\right].$$
	\end{Definition} 
	
	A famous result by Nies (in \cite{N2005}) is that this two notions are equivalent: a set is K-trivial if and only if it is low for $K$. Moreover, Moser and Stephan proved in \cite{MS} that every K-trivial set is shallow.
	
	It is easy to realize that every shallow set remains shallow relatively to any K-trivial oracle.
	
	\begin{Theorem}
		Let $A \in \sequences$ be K-trivial. Then every shallow set is $A$-shallow.
	\end{Theorem}
	
	\begin{proof}
		Let $X \in \sequences$ be shallow. Hence, there is a computable time bound $t$ such that 
		$$\left( \io n \right) \left[ K^t(X \upharpoonright n) =^+ K(X \upharpoonright n)\right].$$
		
		Then, for any such $n$, 
		$$K^{A,t}(X \upharpoonright n) \leq^+ K^t(X \upharpoonright n) =^+ K(X \upharpoonright n) \leq^+ K^A(X \upharpoonright n),$$
		where the last inequality follows because every K-trivial set is low for K.
		Hence, $X$ is $A$-shallow.
	\end{proof}

	It follows from the previous result that the class of deep sets relative to a K-trivial oracle is never larger than the class of deep sets. The following natural question remains open.
	
	\begin{Question}
		Let $A$ be a K-trivial set. Does $A$-depth always strictly imply depth? Or does $A$-depth and depth always coincide? Or does the answer depend on the particular oracle?
	\end{Question}
	
	\bibliographystyle{plain}
	\bibliography{Relativized_depth}

\newcommand{\etalchar}[1]{$^{#1}$}
\begin{thebibliography}{DHMN05}

\bibitem[Ben88]{B}
Charles~H. Bennett.
\newblock Logical depth and physical complexity.
\newblock In {\em The universal {T}uring machine: a half-century survey},
  Oxford Sci. Publ., pages 227--257. Oxford Univ. Press, New York, 1988.

\bibitem[BGH{\etalchar{+}}11]{BGKRS}
Laurent Bienvenu, Peter G\'{a}cs, Mathieu Hoyrup, Cristobal Rojas, and
  Alexander Shen.
\newblock Algorithmic tests and randomness with respect to classes of measures.
\newblock {\em Tr. Mat. Inst. Steklova}, 274(Algoritmicheskie Voprosy Algebry i
  Logiki):41--102, 2011.

\bibitem[BLN10]{BLN}
George Barmpalias, Andrew E.~M. Lewis, and Keng~Meng Ng.
\newblock The importance of {$\Pi^0_1$} classes in effective randomness.
\newblock {\em J. Symbolic Logic}, 75(1):387--400, 2010.

\bibitem[Cha75]{C}
Gregory~J. Chaitin.
\newblock A theory of program size formally identical to information theory.
\newblock {\em J. ACM}, 22(3):329--340, 1975.

\bibitem[CN97]{CN}
Cristian~S. Calude and Andr\'{e} Nies.
\newblock Chaitin {$\Omega$} numbers and strong reducibilities.
\newblock In {\em Proceedings of the {F}irst {J}apan-{N}ew {Z}ealand {W}orkshop
  on {L}ogic in {C}omputer {S}cience ({A}uckland, 1997)}, volume~3, pages
  1162--1166, 1997.

\bibitem[DH10]{DH}
Rodney~G. Downey and Denis~R. Hirschfeldt.
\newblock {\em Algorithmic randomness and complexity}.
\newblock Theory and Applications of Computability. Springer, New York, 2010.

\bibitem[DHMN05]{DHMN}
Rod Downey, Denis~R. Hirschfeldt, Joseph~S. Miller, and Andr\'{e} Nies.
\newblock Relativizing {C}haitin's halting probability.
\newblock {\em J. Math. Log.}, 5(2):167--192, 2005.

\bibitem[G{\'{a}}c86]{G}
P{\'{e}}ter G{\'{a}}cs.
\newblock Every sequence is reducible to a random one.
\newblock {\em Inform. and Control}, 70(2-3):186--192, 1986.

\bibitem[G{\'{a}}c21]{Gacs-notes}
P{\'{e}}ter G{\'{a}}cs.
\newblock Lecture notes on descriptional complexity and randomness.
\newblock 2021.

\bibitem[HKM09]{HKM}
Rupert H\"{o}lzl, Thorsten Kr\"{a}ling, and Wolfgang Merkle.
\newblock Time-bounded kolmogorov complexity and solovay functions.
\newblock In {\em Proceedings of the 34th International Symposium on
  Mathematical Foundations of Computer Science 2009}, pages 392--402, 2009.

\bibitem[HS66]{HS}
Fredrick~C. Hennie and Richard~E. Stearns.
\newblock Two-tape simulations of multitape turing machines.
\newblock {\em JACM}, 4(13):533--546, 1966.

\bibitem[JLL94]{JLL}
David~W. Juedes, James~I. Lathrop, and Jack~H. Lutz.
\newblock Computational depth and reducibility.
\newblock {\em Theor. Comput. Sci.}, 132:37--70, 1994.

\bibitem[MM04]{MM}
Wolfgang Merkle and Nenad Mihailovi\'{c}.
\newblock On the construction of effectively random sets.
\newblock {\em J. Symbolic Logic}, 69(3):862--878, 2004.

\bibitem[MS17]{MS}
Philippe Moser and Frank Stephan.
\newblock Depth, highness and {DNR} degrees.
\newblock {\em Discret. Math. Theor. Comput. Sci.}, 19(4), 2017.

\bibitem[Nie05]{N2005}
Andr\'{e} Nies.
\newblock Lowness properties and randomness.
\newblock {\em Adv. Math.}, 197(1):274--305, 2005.

\bibitem[Nie09]{N2009}
Andr\'{e} Nies.
\newblock {\em Computability and randomness}, volume~51 of {\em Oxford Logic
  Guides}.
\newblock Oxford University Press, Oxford, 2009.

\bibitem[Soa87]{S}
Robert~I. Soare.
\newblock {\em Recursively enumerable sets and degrees}.
\newblock Perspectives in Mathematical Logic. Springer-Verlag, Berlin, 1987.
\newblock A study of computable functions and computably generated sets.

\end{thebibliography}

\end{document}